\documentclass[12pt,leqno]{amsart}
\usepackage{amsfonts,amsthm,amsmath,mathtools,amssymb,comment,xcolor,enumitem}

\usepackage{ascmac}

\usepackage{braket}
\usepackage{stmaryrd}
\usepackage{url}

\usepackage{tikz}
\tikzstyle{v} = [circle, draw, inner sep=1pt, minimum size=3pt, fill=black]

\theoremstyle{plain}
\newtheorem{thm}{Theorem}[section]
\newtheorem{prop}[thm]{Proposition}
\newtheorem{lem}[thm]{Lemma}
\newtheorem{cor}[thm]{Corollary}

\theoremstyle{definition}
\newtheorem{df}[thm]{Definition}
\newtheorem{rem}[thm]{Remark}
\newtheorem{ex}[thm]{Example}
\newtheorem{conj}[thm]{Conjecture}

\newcommand{\FF}{\mathbb{F}}

\newcommand{\NN}{\mathbb{N}}

\DeclareMathOperator{\Aut}{Aut}

\DeclareMathOperator{\Hom}{Hom}

\DeclareMathOperator{\Sym}{Sym}

\begin{document}

\title[Universal graph series]
{Universal graph series, chromatic functions, 
and their index theory}

\author[Miezaki]{Tsuyoshi Miezaki}
\address
{
Faculty of Science and Engineering, 
Waseda University, 
Tokyo 169-8555, Japan\\ (Corresponding author)
}
\email{miezaki@waseda.jp} 
\author[Munemasa]{Akihiro Munemasa}
\address{Graduate School of Information Sciences, Tohoku University, Sendai
980--8579, Japan}
\email{munemasa@tohoku.ac.jp}

\author[Nishimura]{Yusaku Nishimura}
\address{Graduate School of Fundamental Science and Engineering, 
Waseda University}
\email{n2357y@ruri.waseda.jp}

\author[Sakuma]{Tadashi Sakuma}
\address{Faculty of Science, 
Yamagata University,  
Yamagata 990--8560, Japan}
\email{sakuma@sci.kj.yamagata-u.ac.jp}

\author[Tsujie]{Shuhei Tsujie}
\address
{
Department of Mathematics, 
Hokkaido University of Education, 
Asahikawa, Hokkaido 070-8621, Japan.
}
\email{tsujie.shuhei@a.hokkyodai.ac.jp}

\date{\today}

\begin{abstract}
In the present paper, 
we introduce the concept of universal graph series. 
We then present four invariants of graphs and discuss some of 
their properties. 
In particular, one of these invariants is a 
generalization of the chromatic symmetric function and a 
complete invariant for graphs. 
\end{abstract}


\maketitle

\section{Introduction}

The main focus of the present paper is on undirected simple graphs. 
{
Let $G$ and $H$ be graphs. 
We say that a mapping $f:V_G\to V_H$ between 
the sets of their vertices is a homomorphism 
if $\{f(x),f(y)\}\in E_H$ for all $\{x,y\}\in E_G$. 
Let 
$\Hom(G,H)$ denote the set of graph homomorphisms from 
$G$ to $H$. 

\begin{df}
Let $G$ be a finite simple graph. 
The chromatic polynomial of $G$ is defined as 
\[
\chi(G, n) = \sharp \Hom(G, K_n)\ (\forall n \in \NN). 
\]
\end{df}
The chromatic polynomial is not a complete invariant for graphs.
Namely, there exist non-isomorphic graph pairs that have the same chromatic polynomial.
A typical example is a non-isomorphic pair of trees with $m$ vertices. 
Let $T$ be a tree with $m$ vertices. 
Then 
$\chi(T, n) =n(n-1)^{m-1}$. 
Hence, it is natural to ask whether there exists a graph invariant stronger than the chromatic polynomial. 
}

In \cite{Sta95}, 
Stanley defined the concept of chromatic symmetric functions:
\begin{df}[\cite{Sta95}]
Let $G$ be a simple graph and 
$x = \{x_{i}\}_{i \in \mathbb{N}}$ be countably many indeterminates. 
The chromatic symmetric function of $G$ is defined as follows: 
\[
X(G) :=
X(G, x) :=
\sum_{\varphi\in {\rm Hom}(G,K_\NN)}
\prod_{v\in V(G)}
x_{\varphi(v)}, 
\]
where $K_{\mathbb{N}}$ denotes the infinite complete graph on $\mathbb{N}$. 
\end{df}

{
\begin{rem}
This invariant contains all the information of the 
chroamtic polynomials since, 
if $x=\mathbf{1}^n := (\underbrace{1, \ldots, 1}_n
, 0, \ldots), $
then 
\[
X(G, \mathbf{1}^n) = \sum_{\varphi\in \Hom(G,K_n)}1=\sharp \Hom(G,K_n)= \chi(G, n). 
\]
\end{rem}
}


For example, let $K_2$ be a complete graph with two vertices. 
Then \[X(K_2)=\sum_{i<j}2x_ix_j.\]

Stanley conjectured that $X(G)$ is a complete invariant for trees: 
\begin{conj}[\cite{Sta95}]\label{conj:tree}
If $T_1$ and $T_2$ are non-isomorphic trees then 
\[
X(T_1)\neq X(T_2). 
\] 
\end{conj}
However, $X(G)$ is not a complete invariant for graphs. 
Stanley gave the following example \cite{Sta95}. 
Let $V=\{1,2,3,4,5\}$ and 
\begin{align*}
E_1&=\{\{1,2\},\{1,3\},\{2,3\},\{3,4\},\{3,5\},\{4,5\}\},\\
E_2&=\{\{1,2\},\{1,3\},\{1,4\},\{2,3\},\{3,4\},\{4,5\}\}. 
\end{align*}
Then $G_1=(V,E_1)$ and $G_2=(V,E_2)$ are non-isomorphic (See Figure \ref{Fig: G1 and G2}) but 
\[
X(G_1)= X(G_2). 
\]
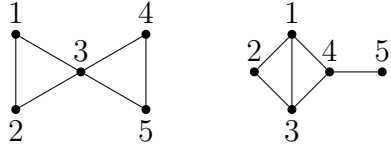
\begin{figure}[t]
\centering

\begin{tikzpicture}
\draw (-.865, .5) node[v](1){} node[above]{1};
\draw (-.865, -.5) node[v](2){} node[below]{2};
\draw (0,0) node[v](3){} node[above]{3};
\draw ( .865, .5) node[v](4){} node[above]{4};
\draw ( .865, -.5) node[v](5){} node[below]{5};
\draw (3)--(1)--(2)--(3)--(4)--(5)--(3);
\end{tikzpicture}
\qquad
\begin{tikzpicture}
\draw (-.5, .5) node[v](1){} node[above]{1};
\draw (-1, 0) node[v](2){} node[above]{2};
\draw (-.5,-.5) node[v](3){} node[below]{3};
\draw (0,0) node[v](4){} node[above]{4};
\draw (.7,0) node[v](5){} node[above]{5};
\draw (1)--(2)--(3)--(4)--(1)--(3);
\draw (4)--(5);
\end{tikzpicture}
\caption{Graphs $G_{1}$ and $G_{2}$} \label{Fig: G1 and G2}

\end{figure}

This gives rise to a natural question: 
is there a generalization of the 
chromatic symmetric function 
which is a complete invariant for graphs?
This paper aims to provide a candidate generalization 
that answers this.

In the present paper, 
we introduce the concept of universal graph series. 
Then using this concept, 
we present four invariants of graphs and discuss some of 
their properties. 
{ 
This concept is motivated by the universal graphs introduced by Rado 
\cite{Rado}. 
}
\begin{df}[\cite{Rado}]
Let $G$ be an infinite graph. 
We say that $G$ is universal if 
it contains every finite graph as an induced subgraph. 
\end{df}
The following is a generalization of the concept of universal graphs: 
\begin{df}
Let $N\subset \NN$ and 
let 
$\{H_n\}_{n\in N}$ be a family of graphs. 
We say that $\{H_n\}_{n\in N}$ is a universal graph series if 
for any simple graph $G$ there exists $n\in N$ such that 
$G$ is an induced subgraph of $H_n$. 
Moreover, the
 universal graph series $\{H_n\}_{n\in N}$ is said to be
induced universal, if 
$H_m$ is an induced subgraph of $H_n$ for every pair
$m,n\in N$ with 
$m\leq n$.
\end{df}
For example, the family of Kneser graphs $\{K_{\NN,k}\}_{k=1}^\infty$ 
is induced universal \cite{{HPW}} and 
that of Paley graphs $\{P(q)\}_{q\in \mathcal{P}}$, 
where $\mathcal{P}$ is the set of all prime powers $q$ with $q\equiv1\pmod{4}$, 
are universal \cite{{BEH}, {BT}, {GS}}. 
{\color{black}
Additionally, the family of Peisert graphs and certain types of generalized Paley graphs are also universal. 
This is because these graphs include any graph of order $n$ as an induced subgraph when the order of the graph is sufficiently large relative to $n$ \cite{A,KW}.
In the present paper, we use two examples as the universal graph series: the family of Kneser graphs and the family of Paley graphs.}
The definitions of these {\color{black}two} graphs will be provided in Sections 4 and 5. 

To define our invariants, we first introduce the $H$-chromatic function, 
defined as follows:

\begin{df}
Let $x_u$ $(u\in V(H))$ be indeterminates. 
We define 
\[
X_H(G) :=
X_H(G, x) :=
\sum_{\varphi \in \Hom(G,H)}
\prod_{v\in V(G)}
x_{\varphi(v)}. 
\]
\end{df}

For a given universal graph series $\{H_n\}_{n\in N}$, 
we define the following invariants for graphs: 
\begin{df}
Let $H=\{H_n\}_{n\in N}$ be a universal graph series and 
$G$ be a simple graph. 
Let $\mathcal{G}$ be the set of all the simple graphs. 
\begin{enumerate}
\item 
We define the universal $H$-chromatic functions of $G$ as follows: 
\[
\{X_{H_n}(G)\}_{n\in N}. 
\]

\item 
For $\mathcal{A}\subset \mathcal{G}$, 
we define the $H$-functional index $I_H(\mathcal{A})$ as
\[
I_H(\mathcal{A})=\min\{t\in N\mid 
\{X_{H_n}(\bullet)\}_{n=1}^{t} 
\text{ is a complete invariant for }\mathcal{A}
\}. 
\]

\item 
We define the $H$-induced index ${i}_{H}(G)$ as
\[
{i}_{H}(G)=\min\{n\in N\mid \text{$G$ is an induced subgraph of $H_n$}\}. 
\]

\item 
We define the $H$-index $\widetilde{i}_{H}(G)$ as
\[
\widetilde{i}_{H}(G)=\min\{n\in N\mid \text{$G$ is a subgraph of $H_n$}\}. 
\]

\end{enumerate}
\end{df}

\begin{rem}
\begin{enumerate}
\item 
Let $H=\{H_n\}_{n\in N}$ be an induced universal graph series. 
Then for all $m,n\in N$ with $m\leq n$, 
$X_{H_{n}}(\bullet)$ is a stronger invariant than 
$X_{H_{m}}(\bullet)$. 
Namely, for any $G,G'\in \mathcal{G}$, 
if $X_{H_{n}}(G)=X_{H_{n}}(G')$ then 
$X_{H_{m}}(G)=X_{H_{m}}(G')$. 
For example, 
as mentioned before, 
$\{K_{\NN,k}\}_{k=1}^\infty$ is induced universal. 
Therefore, 
$X_{K_{\mathbb{N}, k+1}}(\bullet)$ is a stronger invariant than 
$X_{K_{\mathbb{N}, k}}(\bullet)$. 

{
\item 
Let $K=\{K_{\NN,k}\}_{k=1}^{\infty}$ be the Kneser graph series. 
The $K$-(induced) index was considered in \cite{HPW}. 
The definition of the $H$-(indeced) index is motivated by their work. 
}
\end{enumerate}
\end{rem}

Let us explain the main results of the present paper. 
If we have a universal graph series, then we have a 
complete invariant for graphs: 
\begin{thm}\label{thm:main1}
\begin{enumerate}
\item 
Let $H$ be a universal graph. 
Then 
\[
X_{H}(\bullet)
\] 
is a complete invariant for finite graphs. 

\item 
Let $H=\{H_n\}_{n\in N}$ be a 
universal graph series. Then 
\[
\{X_{H_n}(\bullet)\}_{n\in N}
\]
is a complete invariant for finite graphs. 
\end{enumerate}
\end{thm}

For example, let $\{K_{\NN,k}\}_{k\in \NN}$ be the
Kneser graph series. Then 
\[
\{X_{K_{\NN,k}}(\bullet)\}_{k=1}^\infty
\]
is a complete invariant for graphs. 
Since $K_{\mathbb{N}, 1} = K_{\mathbb{N}}$, $X_{K_{\NN,1}}(\bullet)$ is the chromatic symmetric function $X( \bullet)$ and in \cite[Theorem 2.5]{Sta95}, the expansion of $X( \bullet)$ in the basis of power sum symmetric functions was provided. 
The next theorem is a generalization of this formula to 
$K_{\NN,k}$ for $k\in \NN$ 
(see Section \ref{sec:Kneser} for the definitions of $\mathcal{A}_{S}^{(k)}$ and $p_{\lambda}$): 
\begin{thm}\label{power sum expansion}
We have the following: 
\begin{align*}
X_{K_{\NN,k}}(G) 
= \sum_{S \subset E(G)}(-1)^{|S|}\sum_{\lambda \in \mathcal{A}_{S}^{(k)}}p_{\lambda}. 
\end{align*} 
\end{thm}

{
In \cite{HPW}, the upper bound of the Kneser (induced) index was studied. 
}
The following theorem provides upper bounds for some Paley (induced) indices.
Let $P(q)$ denote the Paley graph with order $q$, where $q$ is a prime power
with $q\equiv1\pmod{4}$.
We denote by $K_{k_1,k_2}$ the complete bipartite graph, with parts of size 
$k_1$ and $k_2$.
Let $C_k$ and $P_k$ denote the cycle and path of order $k$, respectively.

\begin{thm}\label{thm:main3}
Let $q$ be an odd prime power.
Then $P=\{P((q^{2})^{3^n})\}_{n=1}^{\infty}$ is an induced universal graph series.
  \begin{enumerate}
\item \label{thm:tri}
  Let $\mathcal{A}_k$ be the set of 
all simple graphs with at most $k$ vertices. 
Then we have the following: 
  \begin{align*}
    I_P(\mathcal{A}_k)&\leq \lceil \log_3\log_{q}((k-1)2^{k-2})\rceil.
  \end{align*}
    \item Let $k_1$ and $k_2$ be any integers satisfying $q^{3^{m-1}}-1\leq k_1\leq k_2\leq q^{3^m}-1$. 
\begin{enumerate}
\item 
Let
    \[\mathcal{H}=\{K_{k_1,k_2},C_{2k_1},P_{k_1+k_2-1}\}\] 
    be a set of graphs.
    Then, for any graph $G\in \mathcal{H}$, we have the following: 
    {\begin{equation*}\label{thm:indexBipartite}
      i_{P}(G)\leq \left\lceil { \log_3\log_q\left((q^{3^m}-3)\begin{pmatrix}
        q^{3^m}-1\\\frac{q^{3^m}-1}{2}
      \end{pmatrix}+3\right)}  \right\rceil.
    \end{equation*}}
\item \label{thm:indexOddCycle}
We have the following: 
{    \begin{align*}
  &i_{P}(C_{2k_1+1})\\
&\leq \left\lceil\log_3\log_q\left({ 2^{q^{3^m}}(q^{3^m}-2)\begin{pmatrix}
q^{3^m}-1\\\frac{q^{3^m}-1}{2}
\end{pmatrix}}+3\right)\right\rceil.
\end{align*}
}        
\end{enumerate}
    \item \label{thm:subInd}
For any integer $k$, we have the following: 
    \begin{equation*}
      \widetilde{i}_{P}(C_k)=\widetilde{i}_{P}(P_k)= 
\left\lceil\log_3\frac{1}{2}(\log_qk)\right\rceil.
    \end{equation*}
  \end{enumerate}
\end{thm}

This paper is organized as follows. 
In Section~\ref{sec:main1}, 
We provide a proof of Theorem \ref{thm:main1}. 
In Section~\ref{sec:com}, 
we present definitions and some basic properties of 
graph homomorphisms and related topics used in this paper. 
In Sections~\ref{sec:Kneser}, 
we offer a proof of Theorem~\ref{power sum expansion}. 
Subsequently, in Sections~\ref{sec:Paley} and~\ref{sec:cycles and paths}, 
we present a proof of Theorem~\ref{thm:main3}. 
Finally, in Section~\ref{sec:rem}, 
we provide concluding remarks and pose some questions for further research.

\section{Proof of Theorem \ref{thm:main1}}\label{sec:main1}
\label{sec:proof of main1}

In order to prove Theorem \ref{thm:main1}, the following lemma is required. 
\begin{lem}[{\cite[p.128 Excercise 11]{GR}}]\label{GR exercise11}
Let $G_{1}$ and $G_{2}$ be finite graphs. 
If $|\Hom(G_{1}, F)| = |\Hom(G_{2}, F)|$ for any finite graph $F$, then $G_{1}$ and $G_{2}$ are isomorphic. 
\end{lem}

\begin{proof}[Proof of Theorem \ref{thm:main1}]

We only give a proof of (1).
The statement (2) can be proved similarly. 

Suppose that $X_{H}(G_{1}) = X_{H}(G_{2})$. 
Let $F$ be a finite graph. 
Since $H$ is a universal graph, $F$ can be regarded as an induced subgraph of $H$. 
Then substituting 
$x_{w}=1$ or $0$ according as $w \in V(F)$ or not
yields $|\Hom(G_{1}, F)| = |\Hom(G_{2}, F)|$. 
Therefore $G_{1}$ and $G_{2}$ are isomorphic by Lemma \ref{GR exercise11}. 
\end{proof}

\section{Functions regarding homomorphisms and weak homomorphisms}\label{sec:com}

In this section, we define weak homomorphisms and give some basic properties of the functions regarding (weak) homomorphisms, which will be used in Section \ref{sec:Kneser}. 


\subsection{Expansion formulas for $X_{H}(G)$ and $W_{H}(G)$}

Let $G$ and $H$ be simple graphs. 
Define 
\begin{align*}
\Hom^{\mathrm{w}}(G,H) &\coloneqq \Set{\varphi \colon V(G) \to V(H) | \begin{array}{l}
\text{If } \{u,v\} \in E(G) \text{ then } \\
\quad \{\varphi(u), \varphi(v)\} \in E(H) \\ 
\quad \text{ or } \varphi(u) = \varphi(v)
\end{array}}. 
\end{align*}
A map in $\Hom^{\mathrm{w}}(G,H)$ is called a {weak homomorphism}.

Define 
\begin{align*}
W_{H}(G):=W_{H}(G, x) &\coloneqq \sum_{\varphi \in \Hom^{\mathrm{w}}(G,H)}\prod_{v \in V(G)}x_{\varphi(v)}. 
\end{align*} 
Both functions $X_{H}(G)$ and $W_{H}(G)$ belong to the ring of formal power series $R_{H} \coloneqq \mathbb{C}\llbracket x_{w} \mid w \in V(H) \rrbracket$. 
The automorphism group $\Aut(H)$ naturally acts on $R_{H}$ and both $X_{H}(G)$ and $W_{H}(G)$ are members of the invariant ring $R_{H}^{\Aut(H)}$.

\begin{prop}\label{weak complement expansion}
$X_{H}(G) = \sum_{S \subset E(G)}(-1)^{|S|}W_{\overline{H}}(G_{S})$, where $G_{S}$ denotes the spanning subgraph of $G$ with edge set $S$ and $\overline{H}$ denotes the complement of $H$. 
\end{prop}
\begin{proof}
\begin{align*}
\sum_{S \subset E(G)}(-1)^{|S|}W_{\overline{H}}(G_{S})
&= \sum_{S \subset E(G)}(-1)^{|S|}\sum_{\varphi \in \Hom^{\mathrm{w}}(G(S),\overline{H})}\prod_{v \in V(G)}x_{\varphi(v)} \\
&= \sum_{\varphi \colon V(G) \to V(H)}\sum_{S \subset E_{\varphi}}(-1)^{|S|} \prod_{v \in V(G)}x_{\varphi(v)}, 
\end{align*}
where 
\begin{align*}
E_{\varphi} &\coloneqq \Set{\{u,v\} \in E(G) | \varphi(u) = \varphi(v) \text{ or } \{\varphi(u), \varphi(v)\} \in E(\overline{H})} \\
&= \Set{\{u,v\} \in E(G) | \{\varphi(u), \varphi(v)\} \not\in E(H)}. 
\end{align*}
Since 
\begin{align*}
\sum_{S \subset E_{\varphi}}(-1)^{|S|} = \begin{cases}
1 & \text{ if } E_{\varphi} = \emptyset. \\
0 & \text{ otherwise. }
\end{cases}
\end{align*}
Thus
\begin{align*}
\sum_{\varphi \colon V(G) \to V(H)}\sum_{S \subset E_{\varphi}}(-1)^{|S|} \prod_{v \in V(G)}x_{\varphi(v)}
= X_{H}(G), 
\end{align*}
which proves the assertion. 
\end{proof}

\section{Power sum expansion of $X_{K_{\mathbb{N},k}}(G,x)$}\label{sec:Kneser}

In this section, we give 
a proof of Theorem \ref{power sum expansion}. 
{
  \begin{df}[Kneser graph]
    Kneser graph $K_{\NN,k}$ is a graph whose vertex set consists of $k$-subsets of $\NN$, and two vertices $A$ and $B$ are adjacent if $A \cap B = \emptyset$.
  \end{df}
}

Note that the automorphism group of Kneser graph $\Aut(K_{\mathbb{N}, k})$ is isomorphic to the symmetric group $S_{\mathbb{N}}$ and $X_{K_{\mathbb{N},k}}(G)$ is invariant under the natural action of $S_{\mathbb{N}}$. 
Define $\Sym^{(k)}$ to 
be the subring of $R_{K_{\mathbb{N},k}}^{S_{\mathbb{N}}}$ consisting of elements of finite degrees. 
Since $X_{K_{\mathbb{N},k}}(G)$ is homogeneous of degree $|V(G)|$, $X_{K_{\mathbb{N},k}}(G)$ belongs to $\Sym^{(k)}$. 
Note that $\Sym^{(1)}$ is the ring of symmetric functions and $X_{K_{\mathbb{N},1}}(G)$ is the chromatic symmetric function of $G$. 

Let $\{I_{1}, \dots, I_{n}\}$ and $\{J_{1}, \dots, J_{n}\}$ 
be two multisets consisting of elements in $\binom{\mathbb{N}}{k}$. 
Define an equivalence relation $\sim$ by $\{I_{1}, \dots, I_{n}\} \sim \{J_{1}, \dots, J_{n}\}$ if there exists $\sigma \in S_{\NN}$ such that 
\[
\{I_{1}, \dots, I_{n}\} = \{\sigma(J_{1}), \dots, \sigma(J_{n})\}
\] 
as multisets. 

\begin{df}
Let $\mathcal{P}_{n}^{(k)}$ denote the equivalence classes of such multisets discussed above and $\mathcal{P}^{(k)} \coloneqq \bigsqcup_{n=1}^{\infty}\mathcal{P}_{n}^{(k)}$. 
\end{df}

Let $\lambda \in \mathcal{P}_{n}^{(k)}$ and $\{I_{1}, \dots, I_{n}\}$ a representative of $\lambda$. 
Then $\lambda$ can be regarded as a $k$-uniform hyper-multigraph on $V_{\lambda} = I_{1} \cup \dots \cup I_{n}$ with edge multiset $E_{\lambda} = \{I_{1}, \dots, I_{n}\}$.
Note that the isomorphism type of this hypergraph is independent of the choice of the representative $\{I_{1}, \dots, I_{n}\}$. 
Also, note that $\mathcal{P}_{n}^{(k)}$ is a finite set for any $n,k$. 

For example, when $k=1$ the graph $\lambda$ is a $1$-uniform hyper-multigraph. 
In this case, $\lambda$ can be identified with the integer partition consisting of multiplicities of hyperedges of $\lambda$. 
Thus $\mathcal{P}^{(1)}$ is the set of integer partitions. 
When $k=2$, $\lambda$ is a $2$-uniform hyper-multigraph. 
Therefore $\mathcal{P}^{(2)}$ is the set of multigraphs without loops or isolated vertices. 

Suppose that $\lambda \in \mathcal{P}^{(k)}$. 
Define the monomial $k$-fold symmetric function $m_{\lambda} = m_{\lambda}^{(k)} \in \Sym^{(k)}$ by 
\begin{align*}
m_{\lambda} \coloneqq \sum_{\{I_{1}, \dots, I_{n}\} \in \lambda}x_{I_{1}} \cdots x_{I_{n}}. 
\end{align*}
Every homogeneous component of $\Sym^{(k)}$ is finite dimensional and the set $\{m_{\lambda}\}_{\lambda}$ forms a linear basis for $\Sym^{(k)}$ over $\mathbb{C}$. 

We say that $\lambda \in \mathcal{P}^{(k)}$ is {connected} if it is connected as a hyper-multigraph. 
The hyper-multigraph $\lambda$ can be decomposed into the disjoint union $\lambda = \lambda_{1} \sqcup \dots \sqcup \lambda_{\ell}$ in the usual manner. 
Then we define the {power sum $k$-fold symmetric function} $p_{\lambda} = p_{\lambda}^{(k)} \in \Sym^{(k)}$ by 
\begin{align*}
p_{\lambda} \coloneqq m_{\lambda_{1}} \cdots m_{\lambda_{\ell}}. 
\end{align*}

Note that when $k=1$, $p_{\lambda}$ is the usual power sum symmetric function. 

\begin{prop}
The set $\{p_{\lambda}\}_{\lambda \in \mathcal{P}^{(k)}}$ forms a linear basis for $\Sym^{(k)}$ over $\mathbb{C}$. 
In particular, $\Sym^{(k)}$ is freely generated as a $\mathbb{C}$-algebra by $\Set{p_{\lambda} \in \mathcal{P}^{(k)} | \lambda \text{ is connected}}$. 
\end{prop}
\begin{proof}
Let $\lambda, \mu \in \mathcal{P}_{n}^{(k)}$. 
Define a partial order $\leq$ on $\mathcal{P}_{n}^{(k)}$ by setting $\mu \leq \lambda$ if $\mu$ is obtained by identifying some vertices of $\lambda$. 

Now, let $\lambda \in \mathcal{P}_{n_{1}}^{(k)}, \mu \in \mathcal{P}_{n_{2}}^{(k)}$ and consider the product $m_{\lambda}m_{\mu}$. 
There exist $c_{\nu} \in \mathbb{Z}_{\geq 0}$ for every $\nu \in \mathcal{P}_{n_{1}+n_{2}}^{(k)}$ such that 
\begin{align*}
m_{\lambda}m_{\mu} = \sum_{\nu \in \mathcal{P}_{n_{1}+n_{2}}^{(k)}}c_{\nu}m_{\nu}. 
\end{align*}
It follows that $c_{\lambda \sqcup \mu} > 0$ and $c_{\nu} = 0$ unless $\nu \leq \lambda \sqcup \mu$. 
Hence
\begin{align*}
m_{\lambda}m_{\mu} = c_{\lambda \sqcup \mu}m_{\lambda \sqcup \mu} + \sum_{\nu < \lambda \sqcup \mu}c_{\nu}m_{\nu}. 
\end{align*}

Therefore for every $\lambda \in \mathcal{P}_{n}^{(k)}$, there exists a positive integer $d_{\mu}$ for every $\mu \in \mathcal{P}_{n}^{(k)}$ such that 
\begin{align*}
p_{\lambda} = d_{\lambda}m_{\lambda} + \sum_{\mu < \lambda}d_{\mu}m_{\mu}. 
\end{align*}

Hence we can order the equivalence classes linearly such that the coefficient matrix of $\{p_{\lambda}\}_{\lambda \in \mathcal{P}^{(k)}}$ for the basis $\{m_{\lambda}\}_{\lambda \in \mathcal{P}^{(k)}}$ is upper triangular. 
Therefore $\{p_{\lambda}\}_{\lambda \in \mathcal{P}^{(k)}}$ is a basis for $\Sym^{(k)}$. 
\end{proof}

\begin{df}
Let $G$ be a connected graph on $n$ vertices. 
We say that $\lambda \in \mathcal{P}_{n}^{(k)}$ is {admissible} by $G$ if there exists a bijection $\varphi \colon V(G) \to E_{\lambda}$ such that $\{u,v\} \in E(G)$ implies $\varphi(u) \cap \varphi(v) \neq \emptyset$. 
Let $\mathcal{A}^{(k)}_{G}$ denote the elements in $\mathcal{P}_{n}^{(k)}$ that is admissible by $G$. 

When $G$ is disconnected and $G = G_{1} \sqcup \dots \sqcup G_{\ell}$ is the decomposition into connected components, define $\mathcal{A}_{G}^{(k)} \coloneqq \mathcal{A}_{G_{1}}^{(k)} \times \dots \times \mathcal{A}_{G_{\ell}}^{(k)}$. 
\end{df}

\begin{lem}\label{weak power sum expansion}
\begin{align*}
W_{\overline{K_{\mathbb{N}, k}}}(G) = \sum_{\lambda \in \mathcal{A}_{G}^{(k)}}p_{\lambda}, 
\end{align*}
where $p_{\lambda} = p_{\lambda_{1}} \cdots p_{\lambda_{\ell}}$ if $\lambda = (\lambda_{1}, \dots, \lambda_{\ell})$. 
\end{lem}
\begin{proof}
First suppose that $G$ is connected. 
Then 
\begin{align*}
W_{\overline{K_{\mathbb{N}, k}}}(G)
= \sum_{\varphi \in \Hom^{\mathrm{w}}(G, \overline{K_{\mathbb{N}, k}})}\prod_{v \in V(G)}x_{\varphi(v)}
= \sum_{\lambda \in \mathcal{A}_{G}^{(k)}}m_{\lambda}
= \sum_{\lambda \in \mathcal{A}_{G}^{(k)}}p_{\lambda}. 
\end{align*}

Next, suppose that $G = G_{1} \sqcup \dots \sqcup G_{\ell}$ is the decomposition into connected components. 
Then the natural bijection 
\begin{align*}
\Hom^{\mathrm{w}}(G, \overline{K_{\mathbb{N}, k}}) \simeq \prod_{i=1}^{\ell} \Hom^{\mathrm{w}}(G_{i}, \overline{K_{\mathbb{N}, k}})
\end{align*}
implies the equality 
\begin{align*}
W_{\overline{K_{\mathbb{N}, k}}}(G)
= \prod_{i=1}^{\ell}W_{\overline{K_{\mathbb{N}, k}}}(G_{i})
= \prod_{i=1}^{\ell}\sum_{\lambda_{i} \in \mathcal{A}_{G_{i}}^{(k)}}p_{\lambda_{i}}
= \sum_{\lambda \in \mathcal{A}_{G}^{(k)}}p_{\lambda}. 
\end{align*}
\end{proof}

\begin{thm}[Restatement of Theorem \ref{power sum expansion}]
\begin{align*}
X_{K_{\mathbb{N},k}}(G) = \sum_{S \subset E(G)}(-1)^{|S|}\sum_{\lambda \in \mathcal{A}_{S}^{(k)}}p_{\lambda}, 
\end{align*}
where $\mathcal{A}_{S}^{(k)}$ stands for $\mathcal{A}_{G_{S}}^{(k)}$ and $G_{S}$ is the spanning subgraph of $G$ with edge set $S$. 
\end{thm}
\begin{proof}
It follows from Proposition \ref{weak complement expansion} and Lemma \ref{weak power sum expansion}. 
\end{proof}

\begin{ex}
Consider the  case $G = P_{3}$ and $k=2$.  
The spanning subgraphs of $G$ are isomorphic to one of $3K_{1}, K_{2}\sqcup K_{1}, P_{3}$. 
We have 
\begin{align*}
\mathcal{A}_{3K_{1}}^{(2)} 
&= \mathcal{A}_{K_{1}}^{(2)} \times \mathcal{A}_{K_{1}}^{(2)} \times \mathcal{A}_{K_{1}}^{(2)} 
= \Set{\left(\begin{tikzpicture}[baseline=2]
\draw (0,0) node[v](1){};
\draw (0,0.4) node[v](2){};
\draw (1)--(2);
\end{tikzpicture} \ \begin{tikzpicture}[baseline=2]
\draw (0,0) node[v](1){};
\draw (0,0.4) node[v](2){};
\draw (1)--(2);
\end{tikzpicture} \ \begin{tikzpicture}[baseline=2]
\draw (0,0) node[v](1){};
\draw (0,0.4) node[v](2){};
\draw (1)--(2);
\end{tikzpicture} \right)},  \\
\mathcal{A}_{K_{2}\sqcup K_{1}}^{(2)} 
&= \mathcal{A}_{K_{2}}^{(2)} \times \mathcal{A}_{K_{1}}^{(2)}
= \Set{\left( \begin{tikzpicture}[baseline=2]
\draw (0,0) node[v](1){};
\draw (0,0.2) node[v](2){};
\draw (0,0.4) node[v](3){};
\draw (1)--(2)--(3);
\end{tikzpicture} \ \begin{tikzpicture}[baseline=2]
\draw (0,0) node[v](1){};
\draw (0,0.4) node[v](2){};
\draw (1)--(2);
\end{tikzpicture} \right), \left( \begin{tikzpicture}[baseline=2]
\draw (0,0) node[v](1){};
\draw (0,0.4) node[v](2){};
\draw (1) to [bend left] (2);
\draw (1) to [bend right] (2);
\end{tikzpicture} \ \begin{tikzpicture}[baseline=2]
\draw (0,0) node[v](1){};
\draw (0,0.4) node[v](2){};
\draw (1)--(2);
\end{tikzpicture} \right)}, \\
\mathcal{A}_{P_{3}}^{(2)} 
&= \Set{\begin{tikzpicture}[baseline=3]
\draw (0,0) node[v](1){};
\draw (0,0.2) node[v](2){};
\draw (0,0.4) node[v](3){};
\draw (0,0.6) node[v](4){};
\draw (1)--(2)--(3)--(4);
\end{tikzpicture}, \ \begin{tikzpicture}[baseline=-5]
\draw (0,0) node[v](1){};
\draw (0,0.3) node[v](2){};
\draw (-0.15,-0.26) node[v](3){};
\draw ( 0.15,-0.26) node[v](4){};
\draw (1)--(2);
\draw (1)--(3);
\draw (1)--(4);
\end{tikzpicture}, \ \begin{tikzpicture}
\draw (0,0) node[v](1){};
\draw (0.3,0) node[v](2){};
\draw (0.15,0.26) node[v](3){};
\draw (1)--(2)--(3)--(1)--cycle;
\end{tikzpicture}, \ \begin{tikzpicture}[baseline=2]
\draw (0,0) node[v](1){};
\draw (0,0.3) node[v](2){};
\draw (0,0.6) node[v](3){};
\draw (1)--(2);
\draw (2) to [bend left] (3);
\draw (2) to [bend right] (3);
\end{tikzpicture}, \ \begin{tikzpicture}[baseline=2]
\draw (0,0) node[v](1){};
\draw (0,0.6) node[v](2){};
\draw (1)--(2);
\draw (1) to [bend left] (2);
\draw (1) to [bend right] (2);
\end{tikzpicture}  }. 
\end{align*}
Therefore 
\begin{align*}
X_{K_{\mathbb{N},2}}(P_{3})
&= \sum_{\lambda \in \mathcal{A}_{3K_{1}}^{(2)}}p_{\lambda} -2\sum_{\lambda \in \mathcal{A}_{K_{2}\sqcup K_{1}}^{(2)}}p_{\lambda} + \sum_{\lambda \in \mathcal{A}_{P_{3}}^{(2)}}p_{\lambda} \\
&= p_{\, \begin{tikzpicture}
\draw (0,0) node[v](1){};
\draw (0,0.4) node[v](2){};
\draw (1)--(2);
\end{tikzpicture} \, \begin{tikzpicture}
\draw (0,0) node[v](1){};
\draw (0,0.4) node[v](2){};
\draw (1)--(2);
\end{tikzpicture} \, \begin{tikzpicture}
\draw (0,0) node[v](1){};
\draw (0,0.4) node[v](2){};
\draw (1)--(2);
\end{tikzpicture}}-2p_{\, \begin{tikzpicture}[baseline=2]
\draw (0,0) node[v](1){};
\draw (0,0.2) node[v](2){};
\draw (0,0.4) node[v](3){};
\draw (1)--(2)--(3);
\end{tikzpicture} \, \begin{tikzpicture}[baseline=2]
\draw (0,0) node[v](1){};
\draw (0,0.4) node[v](2){};
\draw (1)--(2);
\end{tikzpicture}}-2p_{\, \begin{tikzpicture}[baseline=2]
\draw (0,0) node[v](1){};
\draw (0,0.4) node[v](2){};
\draw (1) to [bend left] (2);
\draw (1) to [bend right] (2);
\end{tikzpicture} \, \begin{tikzpicture}[baseline=2]
\draw (0,0) node[v](1){};
\draw (0,0.4) node[v](2){};
\draw (1)--(2);
\end{tikzpicture}} + p_{\, \begin{tikzpicture}[baseline=3]
\draw (0,0) node[v](1){};
\draw (0,0.2) node[v](2){};
\draw (0,0.4) node[v](3){};
\draw (0,0.6) node[v](4){};
\draw (1)--(2)--(3)--(4);
\end{tikzpicture}} +p_{ \begin{tikzpicture}[baseline=-5]
\draw (0,0) node[v](1){};
\draw (0,0.3) node[v](2){};
\draw (-0.15,-0.26) node[v](3){};
\draw ( 0.15,-0.26) node[v](4){};
\draw (1)--(2);
\draw (1)--(3);
\draw (1)--(4);
\end{tikzpicture}} + p_{\begin{tikzpicture}
\draw (0,0) node[v](1){};
\draw (0.3,0) node[v](2){};
\draw (0.15,0.26) node[v](3){};
\draw (1)--(2)--(3)--(1)--cycle;
\end{tikzpicture}} + p_{\begin{tikzpicture}[baseline=2]
\draw (0,0) node[v](1){};
\draw (0,0.6) node[v](2){};
\draw (1)--(2);
\draw (1) to [bend left] (2);
\draw (1) to [bend right] (2);
\end{tikzpicture}}. 
\end{align*}
\end{ex}

\begin{ex}
Consider the  case $G = K_{3}$ and $k=2$. 
The spanning subgraphs of $G$ are isomorphic to one of $3K_{1}, K_{2}\sqcup K_{1}, P_{3}, K_{3}$. 
Since 
\begin{align*}
\mathcal{A}_{K_{3}}^{(2)} 
&= \Set{\begin{tikzpicture}[baseline=-5]
\draw (0,0) node[v](1){};
\draw (0,0.3) node[v](2){};
\draw (-0.15,-0.26) node[v](3){};
\draw ( 0.15,-0.26) node[v](4){};
\draw (1)--(2);
\draw (1)--(3);
\draw (1)--(4);
\end{tikzpicture}, \ \begin{tikzpicture}
\draw (0,0) node[v](1){};
\draw (0.3,0) node[v](2){};
\draw (0.15,0.26) node[v](3){};
\draw (1)--(2)--(3)--(1)--cycle;
\end{tikzpicture}, \ \begin{tikzpicture}[baseline=2]
\draw (0,0) node[v](1){};
\draw (0,0.3) node[v](2){};
\draw (0,0.6) node[v](3){};
\draw (1)--(2);
\draw (2) to [bend left] (3);
\draw (2) to [bend right] (3);
\end{tikzpicture}, \ \begin{tikzpicture}[baseline=2]
\draw (0,0) node[v](1){};
\draw (0,0.6) node[v](2){};
\draw (1)--(2);
\draw (1) to [bend left] (2);
\draw (1) to [bend right] (2);
\end{tikzpicture}  }
\end{align*}
we have 
\begin{align*}
X_{K_{\mathbb{N},2}}(K_{3})
&= \sum_{\lambda \in \mathcal{A}_{3K_{1}}^{(2)}}p_{\lambda} -3\sum_{\lambda \in \mathcal{A}_{K_{2}\sqcup K_{1}}^{(2)}}p_{\lambda} + 3\sum_{\lambda \in \mathcal{A}_{P_{3}}^{(2)}}p_{\lambda} - \sum_{\lambda \in \mathcal{A}_{K_{3}}^{(2)}}p_{\lambda} \\
&= p_{\, \begin{tikzpicture}
\draw (0,0) node[v](1){};
\draw (0,0.4) node[v](2){};
\draw (1)--(2);
\end{tikzpicture} \, \begin{tikzpicture}
\draw (0,0) node[v](1){};
\draw (0,0.4) node[v](2){};
\draw (1)--(2);
\end{tikzpicture} \, \begin{tikzpicture}
\draw (0,0) node[v](1){};
\draw (0,0.4) node[v](2){};
\draw (1)--(2);
\end{tikzpicture}}-3p_{\, \begin{tikzpicture}[baseline=2]
\draw (0,0) node[v](1){};
\draw (0,0.2) node[v](2){};
\draw (0,0.4) node[v](3){};
\draw (1)--(2)--(3);
\end{tikzpicture} \, \begin{tikzpicture}[baseline=2]
\draw (0,0) node[v](1){};
\draw (0,0.4) node[v](2){};
\draw (1)--(2);
\end{tikzpicture}}-3p_{\, \begin{tikzpicture}[baseline=2]
\draw (0,0) node[v](1){};
\draw (0,0.4) node[v](2){};
\draw (1) to [bend left] (2);
\draw (1) to [bend right] (2);
\end{tikzpicture} \, \begin{tikzpicture}[baseline=2]
\draw (0,0) node[v](1){};
\draw (0,0.4) node[v](2){};
\draw (1)--(2);
\end{tikzpicture}} + 3p_{\, \begin{tikzpicture}[baseline=3]
\draw (0,0) node[v](1){};
\draw (0,0.2) node[v](2){};
\draw (0,0.4) node[v](3){};
\draw (0,0.6) node[v](4){};
\draw (1)--(2)--(3)--(4);
\end{tikzpicture}} +2p_{ \begin{tikzpicture}[baseline=-5]
\draw (0,0) node[v](1){};
\draw (0,0.3) node[v](2){};
\draw (-0.15,-0.26) node[v](3){};
\draw ( 0.15,-0.26) node[v](4){};
\draw (1)--(2);
\draw (1)--(3);
\draw (1)--(4);
\end{tikzpicture}} + 2p_{\begin{tikzpicture}
\draw (0,0) node[v](1){};
\draw (0.3,0) node[v](2){};
\draw (0.15,0.26) node[v](3){};
\draw (1)--(2)--(3)--(1)--cycle;
\end{tikzpicture}} + 2p_{\begin{tikzpicture}[baseline=2]
\draw (0,0) node[v](1){};
\draw (0,0.6) node[v](2){};
\draw (1)--(2);
\draw (1) to [bend left] (2);
\draw (1) to [bend right] (2);
\end{tikzpicture}}. 
\end{align*}
\end{ex}

\section{Evaluation of the Paley-induced index}\label{sec:Paley}
In this section, we first provide a trivial upper bound for the Paley-induced index. 
Then, we give proofs of Theorem \ref{thm:main3} \eqref{thm:tri} and \eqref{thm:indexBipartite}. 

{
The lower and upper bounds of the Kneser-induced index of paths, cycle graphs, and hypercube graphs are shown in \cite{HPW}.
Therefore, we investigate the induced index using another universal graph.

\begin{df}[Paley graph]
  {\color{black} Let $q$ be a prime power such that $q\equiv 1\pmod{4}$.}
  A Paley graph $P(q)$ is a graph whose vertex set consists of the elements of a finite field of order $q$, where two vertices $u$ and $v$ are adjacent if $u-v \in (\FF_q^*)^2$.
\end{df}

It is known that if $q$ is sufficiently large relative to $n$, $P(q)$ contains any graph with order $n$.
Hence, any infinite sequence of Paley graphs forms a universal graph series.
We then investigate the Paley-induced index.
}

\subsection{A trivial upper bound for the Paley-induced index}
In this subsection, 
we provide a trivial upper bound for the Paley-induced index. 

Let $\FF_q$ be a finite field with order $q$ and $P(q)$ be 
the Paley graph of order $q$. 
In the following, we always denote by $P_n(q,k)$ the Paley graph $P(q^{k^n})$. 

\begin{prop}\label{prop:inducedArray}
Let $k$ be any positive integer. 
Then $P_n(q,2k+1)$ is an induced subgraph of $P_{n+1}(q,2k+1)$. 
\end{prop}

\begin{proof}
  For the sake of simplicity, we denote $q^{(2k+1)^n}$ as $q_0$.
  We consider the induced subgraph of $P_{n+1}(q,2k+1)$ constructed by $\FF_{q_0}$.
  This graph can be denoted as 
  \[
    Cay(\FF_{q_0},(\FF_{q_0^{2k+1}}^*)^2\cap \FF_{q_0}),
  \]
  where $Cay(G,S)$ represents the Cayley graph made by the group $G$ and its subset $S$.
  Therefore, if 
  \[
    (\FF_{q_0^{2k+1}}^*)^2\cap \FF_{q_0}= (\FF_{q_0}^*)^2 
  \]
  then this graph is $P_n(q,2k+1)$.
  Because $\FF_{q_0^{2k+1}}$ is an odd-degree extension of $\FF_{q_0}$, the above equality holds.
  Hence, the induced subgraph of $P_{n+1}(q,2k+1)$ constructed by $\FF_{q_0}$ is $P_n(q,2k+1)$.
\end{proof}

From Proposition \ref{prop:inducedArray}, 
$\{P_n(q,2k+1)\}_{n\in \NN}$ is induced universal. 
In the following, we consider the case $P=\{P_n(q^2,3)\}_{n\in \NN}$. 
Then, we can define the Paley-induced index $i_{P}$ as follows: 
\[
i_{P}(G)
=\min\{n\in \NN\mid \mbox{$G$ is an induced subgraph of $P_n(q^2,3)$}\}. 
\]


There exists a trivial upper bound for this invariant using the number of vertices in the graph.

\begin{thm}\label{thm:upperB}
  Let $G$ be a graph with order $k$. Then,
  \begin{align*}
    i_{P}(G)&\leq \lceil \log_3\log_{q}((k-1)2^{k-2})\rceil.
  \end{align*}
\end{thm}

\begin{proof}
  From Theorem 7.19 of \cite{MR3821579}, any graph $G$ with order $k$ is an induced subgraph of $P(q_0)$ if the following inequality holds:
  \[
  q_0>((k-1)2^{k-2})^2.  
  \]
  Therefore, $G$ is an induced subgraph of $P_n(q^2,3)$ for any $n$ which satisfies 
  \[
    (q^2)^{3^n}>((k-1)2^{k-2})^2 \Leftrightarrow n>\log_3\log_q((k-1)2^{k-2}).
  \]
  Especially, since $\lceil \log_3\log_q((k-1)2^{k-2})\rceil>\log_3\log_q((k-1)2^{k-2})$, $P_n(q^2,3)$ contains $G$ when $n=\lceil \log_3\log_q((k-1)2^{k-2})\rceil$.
  Hence,
  \[
    i_{P}(G)\leq \lceil \log_3\log_q((k-1)2^{k-2})\rceil.  
  \]
\end{proof}

\begin{proof}[Proof of Theorem \ref{thm:main3} \eqref{thm:tri}]
Let $\mathcal{A}_k$ be the set of 
all simple graphs with at most $k$ vertices. 
Theorem~\ref{thm:upperB} implies the upper bound for $I_P(\mathcal{A}_k)$. 
Similarly to the discussion in the proof of Theorem \ref{thm:main1}, 
for any graph $G\in\mathcal{A}_k$, 
$\{|\Hom(G,F)|\}_{F\in\mathcal{A}_k}$ is also a complete invariant. 
Therefore, if there exists a graph $H$ such that $H$ contains 
all graphs in $\mathcal{A}_k$ as induced subgraphs, 
then $X_H(G)$ is a complete invariant. 
Hence, from Theorem~\ref{thm:upperB}, 
\begin{align*}
I_P(\mathcal{A}_k)&\leq \lceil \log_3\log_{q}((k-1)2^{k-2})\rceil. 
\end{align*}
\end{proof}

\subsection{Lemmas for a non-trivial upper bound for some Paley-induced indices}\label{subsec:lem}
In this subsection, we provide lemmas to prove the statement 
\eqref{thm:indexBipartite} in Theorem \ref{thm:main3}. 
First, we introduce the theorems related to the character sums used in the proofs of the lemmas. 
\begin{thm}[\cite{10.1073/pnas.34.5.204}]\label{thm:weil}
  Let $f$ be a polynomial in $\FF_q[x]$ of degree $d$ and let $\sigma$ be a multiplicative character of $\FF^*_{q}$ with order $2$.  
  If there does not exist any polynomial $g$ in $\FF_q[x]$ such that $f=g^2$, then
  \[
    \left|\sum_{v\in \FF_q}\sigma(f(v))\right| \leq (d-1)\sqrt{q}.
  \]
\end{thm}

{
\begin{thm}[\cite{A}]\label{thm:Ananchuen}
  Let $f$ be a polynomial in $\FF_q[x]$ of degree $d$ and let $\sigma$ be a multiplicative character of $\FF^*_{q}$ with order $2$.  
  If $d$ is even and there does not exist any polynomial $g$ in $\FF_q[x]$ such that $f=g^2$, then
  \[
    \left|\sum_{v\in \FF_q}\sigma(f(v))\right| \leq 1+(d-2)\sqrt{q}.
  \]
\end{thm}
}

Afterwards, we denote $K_{k_1,k_2}$ as the complete bipartite graph, where $k_1$ and $k_2$ represent the sizes of the independent sets, respectively,
and $C_k$ and $P_k$ denote the cycle and path with order $k$, respectively.
Let $p$ be a prime number and let $m$ and $n$ be any integers with $m<n$. 
We define $q:=p^{3^m}$ and $q_0:=(p^2)^{3^n}$.

\begin{lem}\label{lem:bipartite} 
  If $q_0$ satisfies the following inequality: 
  {
  \begin{align*}
  q_0> \left((q-3)\begin{pmatrix}
    q-1\\\frac{q-1}{2}
  \end{pmatrix}+3\right)^2
\end{align*}
}then $K_{q-1,q-1}$ is an induced subgraph of $P(q_0)$, { when $q\geq 5$.} 
\end{lem}

\begin{proof}
  Since $\FF_{q_0}$ is an even-degree extension of $\FF_{q}$, $\FF_{q}\subset (\FF_{q_0}^*)^2$. 

  Let $x$ be any element of quadratic non-residues of $\FF_{q_0}$, and define $A$ as the vertex set of $P(q_0)$ as follows:
  \[
    A\coloneqq\{ax \mid a \in \FF_q\}.    
  \]
  From the definition of $A$, for any two elements in $A$, the difference between them is $ax$, where $a$ is some element in $\FF_q$.
  Therefore, the induced subgraph of $P(q_0)$ constructed by $A$ is an independent set with order $q$.
  Next, we assume that there exists a vertex $y$ such that $y\in V\setminus A$ and $y$ is adjacent to all elements of $A\setminus\{0\}$.
  Similarly, we define $B$ as the vertex set of $P(q_0)$ as follows:
  \[
    B\coloneqq\{by\mid b\in \FF_q\}.  
  \]
  Because $y$ is not adjacent to $0$, $y$ is a quadratic non-residue, and an induced subgraph of $P(q_0)$ constructed by $B$ is also an independent set with order $q$.
  Furthermore, $by-ax=b(y-b^{-1}ax)$ and $y-b^{-1}ax$ is a quadratic residue because of the definition of $y$.
  This implies that for any element in $B$, it is adjacent to all elements of $A\setminus\{0\}$.
  Therefore, the induced subgraph constructed by $(A\cup B)\setminus\{0\}$ is $K_{q-1,q-1}$.
  Additionally, $y$ exists if and only if a vertex, denoted as $y'=y+x$, also exists which is adjacent to all elements of $A\setminus \{x\}$.
  Hence, we consider the condition of $q_0$ under which a vertex $y'$ exists.
  
  Let $D=\{x\},C=A\setminus D$ and $\sigma$ be a multiplicative character of $\FF^*_{q_0}$ with order $2$, where we define $\sigma(0)=0$.
  We consider the function $\tau$ as follows:
  \[
    \tau(v)=\prod_{c\in C}(1+\sigma(v-c))\prod_{d\in D}(1-\sigma(v-d)).  
  \]
  When $\tau(v)=2^q$, then $v$ is adjacent to all vertices of $C$ and not adjacent to any vertices of $D$.
  Here, if $v\in \FF_{q_0}\setminus A$, then 
  \[
    \tau(v)\in
      \{0, 2^q\}. 
 \]
  Therefore, the condition for the existence of the desired vertex $y'$ is the following inequality: 
  \[
    \sum_{v\in \FF_{q_0}\setminus A}\tau(v)>0. 
  \]
  Additionally, in the case $v\in A$, $\tau(v)$ is $0$.
  Therefore, we consider the bound for $q_0$ which satisfies the following inequality:
  \[
  \sum_{v\in \FF_{q_0}}\tau(v)>0.
  \]

  Expanding the product that defines $\tau(v)$, we obtain
  \[
    \sum_{v\in \FF_{q_0}}\tau(v)=q_0+\sum_{\emptyset\neq S\subset A}(-1)^{|S\cap D|}\sum_{v\in\FF_{q_0}}\sigma(f_S(v)),
  \]
  where $S$ represents any subset of $A$, and $f_S(v)=\prod_{s\in S}(v-s)$.
  Let $\mathcal{S}$ be the power set of $A$.
  Then, we can define a bijection $\varphi_a$ using some element $a\in \FF_q$ as follows:
  \begin{align*}
    \varphi_a: \mathcal{S}\longrightarrow \mathcal{S}; S\mapsto aS.
  \end{align*}
  Additionally, because $\sigma(a)=1$,
  \begin{equation*}
    \begin{split}
      \sum_{v\in \FF_{q_0}}\sigma(f_{aS}(v))&=\sum_{v\in \FF_{q_0}}\sigma\left(\prod_{x\in aS}(v-x)\right)\\
    &=\sum_{v\in \FF_{q_0}}\sigma(a)^{|S|}\sigma\left(\prod_{x\in S}(a^{-1}v-x)\right)\\
    &=\sum_{v\in \FF_{q_0}}\sigma\left(\prod_{x\in S}(v-x)\right)\\
    &=\sum_{v\in\FF_{q_0}}\sigma(f_{S}(v)).
    \end{split}
  \end{equation*}
  Therefore,
  \begin{align*}
    \sum_{\emptyset\neq S\subset A}(-1)^{|S\cap D|}\sum_{v\in \FF_{q_0}}\sigma(f_S(v))&=\sum_{\emptyset\neq S\subset A}(-1)^{|(aS)\cap (aD)|}\sum_{v\in \FF_{q_0}}\sigma(f_{S}(v))  \\
    &=\sum_{\emptyset\neq S\subset A}(-1)^{|S\cap \{ax\}|}\sum_{v\in \FF_{q_0}}\sigma(f_{a^{-1}S}(v))\\
    &=\sum_{\emptyset\neq S\subset A}(-1)^{|S\cap \{ax\}|}\sum_{v\in \FF_{q_0}}\sigma(f_{S}(v)).
  \end{align*}
From the above,
\begin{align*}
  \sum_{\emptyset\neq S\subset A}&(-1)^{|S\cap D|}\sum_{v\in \FF_{q_0}}\sigma(f_S(v))\\
&=\sum_{\emptyset\neq S\subset A}\sum_{a\in \FF_q^*}\frac{(-1)^{|S\cap \{ax\}|}}{q-1}\sum_{v\in \FF_{q_0}}\sigma(f_S(v))\\
  &=\sum_{\emptyset\neq S\subset A\setminus \{0\}}\sum_{a\in \FF_q^*}\frac{(-1)^{|S\cap \{ax\}|}}{q-1}\sum_{v\in \FF_{q_0}}(\sigma(f_S(v))+\sigma(f_{S\cup\{0\}}(v))).
\end{align*}
Furthermore, the following equality holds when $S\subset A\setminus \{0\}$:
\begin{align*}
  \sum_{a\in \FF_q^*}(-1)^{|S\cap \{ax\}|}&=|S|(-1)^{1}+q-1-|S|(-1)^0\\
  &=q-1-2|S|.
\end{align*}
{
  Since either $|S|$ or $|S\cup\{0\}|$ is always even, from Theorem \ref{thm:weil} and Theorem \ref{thm:Ananchuen},
\begin{align*}
  \sum_{v\in \FF_{q_0}}(\sigma(f_S(v))+\sigma(f_{S\cup\{0\}}(v)))&\geq -\left|\sum_{v\in \FF_{q_0}}(\sigma(f_S(v))+\sigma(f_{S\cup\{0\}}(v)))\right|\\
  &\geq -(1+2\sqrt{q_0}(|S|-1))
\end{align*}
}
Therefore,
{
\begin{align*}
  \sum_{v\in \FF_{q_0}}\tau(v)&\geq q_0-\sum_{\emptyset\neq S\subset A\setminus\{0\}}\frac{|q-1-2|S||}{q-1}(1+2(|S|-1)\sqrt{q_0})\label{eq:bi}.
\end{align*}
}
On the other hand,

{
  \begin{align*}
  \sum_{\emptyset\neq S\subset A\setminus\{0\}}
&\frac{|q-1-2|S||}{q-1}(1+2(|S|-1)\sqrt{q_0})\\
&=\sum_{i=1}^{q-1}\frac{|q-1-2i|}{q-1}(1+2(i-1)\sqrt{q_0})\begin{pmatrix}
    q-1\\
    i
  \end{pmatrix}\\
  &=\left(\sum_{i=1}^{\frac{q-1}{2}}-\sum_{i=\frac{q+1}{2}}^{q-1}\right)\left(\frac{q-1-2i}{q-1}(1+2(i-1)\sqrt{q_0})\begin{pmatrix}
    q-1\\i
  \end{pmatrix}\right),
\end{align*}
}
where $(\sum_{i=1}^{x}-\sum_{i=x+1}^{y})f(i)=(\sum_{i=1}^{x}f(i))-(\sum_{i=x+1}^{y}f(i))$.
From the properties of binomial coefficients, 
{
  \begin{align*}
  \frac{2(q-1-2i)(i-1)}{q-1}\begin{pmatrix}
    q-1\\i
  \end{pmatrix}&=\left(\begin{pmatrix}
    q-1\\i
  \end{pmatrix}-2\begin{pmatrix}
    q-2\\i-1
  \end{pmatrix}\right)2(i-1)\\
  &=2\left(
    -\begin{pmatrix}
      q-1\\i
    \end{pmatrix}+
    (q-1)\begin{pmatrix}
      q-2\\i-1
    \end{pmatrix}-2(q-2)\begin{pmatrix}
      q-3\\i-2
    \end{pmatrix}
  \right).
\end{align*}
}
Furthermore, 
\begin{align*}
  \sum_{i=1}^{\frac{q-1}{2}}\begin{pmatrix}
    q-1\\
    i
  \end{pmatrix}-\sum_{i=\frac{q+1}{2}}^{q-1}\begin{pmatrix}
    q-1\\
    i
  \end{pmatrix}&=\begin{pmatrix}
    q-1\\
    \frac{q-1}{2}
  \end{pmatrix}-1,\\
  \sum_{i=1}^{\frac{q-1}{2}}\begin{pmatrix}
    q-2\\
    i-1
  \end{pmatrix}-\sum_{i=\frac{q+1}{2}}^{q-1}\begin{pmatrix}
    q-2\\
    i-1
  \end{pmatrix}&=0,\\
  \sum_{i=2}^{\frac{q-1}{2}}\begin{pmatrix}
    q-3\\
    i-2
  \end{pmatrix}-\sum_{i=\frac{q+1}{2}}^{q-1}\begin{pmatrix}
    q-3\\
    i-2
  \end{pmatrix}&=-\begin{pmatrix}
    q-3\\
    \frac{q-3}{2}
  \end{pmatrix}.
\end{align*}
Hence,
{
\begin{align*}
  &\left(\sum_{i=1}^{\frac{q-1}{2}}-\sum_{i=\frac{q+1}{2}}^{q}\right)\frac{q-1-2i}{q-1}(1+2(i-1)\sqrt{q_0})\begin{pmatrix}
    q-1\\i
  \end{pmatrix}\\
&=\left(\begin{pmatrix}
  q-1\\
  \frac{q-1}{2}
\end{pmatrix}-1\right)+
2\sqrt{q_0}\left(-\begin{pmatrix}
    q-1\\\frac{q-1}{2}
  \end{pmatrix}
  +1+2(q-2)\begin{pmatrix}
    q-3\\\frac{q-3}{2}
  \end{pmatrix}\right)\\
  &=\sqrt{q_0}\left((q-3)\begin{pmatrix}
    q-1\\\frac{q-1}{2}
  \end{pmatrix}+2\right)+\begin{pmatrix}
    q-1\\\frac{q-1}{2}
  \end{pmatrix}-1.
\end{align*}
}
Based on the above,

{
  \begin{align*}
  \sum_{v\in \FF_{q_0}}\tau(v)&\geq q_0-\sqrt{q_0}\left((q-3)\begin{pmatrix}
    q-1\\\frac{q-1}{2}
  \end{pmatrix}+2\right)-\begin{pmatrix}
    q-1\\\frac{q-1}{2}
  \end{pmatrix}+1.
\end{align*}
}
The desired vertex exists if the right-hand side of this inequality is greater than $0$.
Especially, when $q \geq 5$, if $q_0$ satisfies the following inequality:
{
  \begin{align*}
  q_0> \left((q-3)\begin{pmatrix}
    q-1\\\frac{q-1}{2}
  \end{pmatrix}+3\right)^2
\end{align*}
}
then the desired vertex exists and 
$K_{q-1,q-1}$ is an induced subgraph of $P(q_0)$. 
\end{proof}

\begin{cor}\label{cor:bipartite}
  For any integers $k_1$ and $k_2$ such that $0<k_1\leq k_2<q-1$, if $q_0$ satisfies the following inequality: 
  {
  \begin{align*}
  q_0> \left((q-3)\begin{pmatrix}
    q-1\\\frac{q-1}{2}
  \end{pmatrix}+3\right)^2
\end{align*}
} then $K_{k_1,k_2}$ is an induced subgraph of $P(q_0)${, where $q\geq 5$}. 
\end{cor}

\begin{proof}
  $K_{k_1,k_2}$ is an induced subgraph of $K_{q-1,q-1}$, and from Lemma \ref{lem:bipartite}, $K_{q-1,q-1}$ is an induced subgraph of $P(q_0)$ under this condition.
\end{proof}

\begin{lem}\label{lem:cycle1}
  If $q_0$ satisfies the following inequality: 
  {
  \begin{align*}
  q_0> \left((q-3)\begin{pmatrix}
    q-1\\\frac{q-1}{2}
  \end{pmatrix}+3\right)^2
\end{align*}
}
then $C_{2(q-1)}$ is an induced subgraph of $P(q_0)${, where $q\geq 5$}. 
\end{lem}

\begin{proof}
  In the same manner of Lemma \ref{lem:bipartite}, let $x$ be some 
quadratic non-residue in $\FF_{q_0}$, and define $A:=\{ax \mid a \in \FF_q\}$ as the vertex subset of $P(q_0)$.
  The induced subgraph constructed by $A$ forms an independent set with order $q$.  

  Let $\alpha$ be a primitive root of $\FF_q$, and we assume the existence of a vertex $y$ which is adjacent to $x$ and $\alpha x$, and not adjacent to other vertices of $A$.
  From the property of $A$, if such a vertex exists then $y\in \FF_{q_0}\setminus A$.

  We also define $B:=\{ay\mid a\in \FF_q\}$ as the vertex subset of $P(q_0)$.
  Similarly, $B$ is also an independent set with order $q$.

  Because $y$ is adjacent to only two vertices in $A$, $\alpha^i y$ is also adjacent to only two vertices in $A$, $\alpha^{i}x$ and $\alpha^{i+1}x$.
  Therefore, the induced subgraph constructed by $(A\cup B)\setminus\{0\}$ forms $C_{2(q-1)}$.
  Using the same argument as in Lemma \ref{lem:bipartite}, the requirement for the existence of $y$ is exactly the same as the requirement for the existence of $y'=y-x$, 
  which is only adjacent to $0$ and $(\alpha-1)x$ in $A$.
  We consider the condition of $y'$.

  Let $C=\{0,(\alpha-1)x\}$ and $D=A\setminus C$.
  Similar to Lemma \ref{lem:bipartite}, we consider 
  \[
    \tau(v)=\prod_{c\in C}(1+\sigma(v-c))\prod_{d\in D}(1-\sigma(v-d)).  
  \]
  Then, the desired vertex $y'$ exists if 
  \[
    \sum_{v\in \FF_{q_0}\setminus A}\tau(v)>0.
  \]
  Since for any vertex $v\in A, \tau(v)=0$, we can rephrase the above inequality as follows:
  \[
    \sum_{v\in \FF_{q_0}}\tau(v)=q_0+\sum_{\emptyset\neq S\subset A}(-1)^{|S\cap D|}\sum_{v\in \FF_{q_0}}\sigma(f_S(v))>0. 
  \]
  Similar to the discussion of Lemma \ref{lem:bipartite}, 
  \begin{align*}
    \sum_{\emptyset\neq S\subset A}(-1)^{|S\cap D|}&\sum_{v\in \FF_{q_0}}\sigma(f_S(v))\\
&=\sum_{\emptyset\neq S\subset A}\sum_{a\in \FF_q^*}\frac{(-1)^{|S\cap aD|}}{q-1}\sum_{v\in \FF_{q_0}}\sigma(f_S(v))\\
    &=\sum_{\emptyset\neq S\subset A\setminus \{0\}}\frac{(q-1-2|S|)}{q-1}\sum_{v\in\FF_{q_0}}(\sigma(f_S(v))+\sigma(f_{S\cup \{0\}}(v))).
  \end{align*}
  Therefore, from Theorem \ref{thm:weil} and Theorem \ref{thm:Ananchuen},
{  \[
    \sum_{v\in \FF_{q_0}}\tau(v)\geq q_0-\sqrt{q_0}\sum_{\emptyset\neq S\subset A\setminus \{0\}}\frac{|q-1-2|S||}{q-1}(1+2(|S|-1)\sqrt{q_0}).
  \]
}  This inequality is exactly the same as inequality \eqref{eq:bi} in Lemma \ref{lem:bipartite}.
  Hence, we obtain an equivalent inequality to Lemma \ref{lem:bipartite}, and the proof is complete.
\end{proof}

\begin{cor}\label{cor:cycle1}
Let $k$ be an integer such that $0<k<q-1$. 
If $q_0$ satisfies the following inequality: 
{
  \begin{align*}
  q_0> \left((q-3)\begin{pmatrix}
    q-1\\\frac{q-1}{2}
  \end{pmatrix}+3\right)^2
\end{align*}
}then $C_{2k+2}$ is an induced subgraph of $P(q_0)${, where $q\geq 5$}. \end{cor}

\begin{proof}
  Let $\alpha$ be a primitive root of $\FF_{q}$.
  From Lemma~\ref{lem:cycle1}, there exist vertex sets $A=\{ax \mid a \in \FF_q\}$ and $B=\{ay \mid a \in \FF_q\}$ in $P(q_0)$ under this condition,
  where $y$ is only adjacent to $x$ and $\alpha x$ in $A$.
  Additionally, we also know there is a vertex $y'=y-x$ which is only adjacent to $0$ and $(\alpha-1)x$ in $A$.
  We consider the adjacency between $y'$ and $B$.
  When $a\neq 1$, because
  \[
    y'-ay=(1-a)(y-(1-a)^{-1}x)
  \]
  and $1-a$ is a quadratic residue, $y'$ is adjacent to $ay$ if and only if $a=1-\alpha^{-1}$ or $a=0$.
  This implies that $y'$ is only adjacent to $0,(\alpha-1)x$, and $\alpha^{-1}(\alpha-1)y$ in $A\cup B$.
  Especially, $z=(\alpha-1)^{-1}y'$ is only adjacent to $0,x$ and $\alpha^{-1}y$, and $\alpha^kz$ is only adjacent to $0,\alpha^{k} x$, and $\alpha^{k-1}y$.
  Furthermore, because $z$ is a quadratic residue, $z$ is adjacent to $\alpha^k z$.
  Hence, the induced subgraph obtained by 
  \[
    \{x,y,\alpha x,\alpha y,\alpha^2 x,\alpha^2 y,\ldots,\alpha^{k-2} x,\alpha^{k-2} y,\alpha^{k-1} x,\alpha^{k-1} y,\alpha^k z,z\}  
  \]
  is $C_{2k+2}$.
  Therefore, if $q_0$ satisfies the inequality of Lemma~\ref{lem:cycle1} then $C_{2k+2}$ is an induced subgraph of $P(q_0)$.
\end{proof}


\begin{cor}\label{cor:path1}
Let $k$ be an integer such that $k < 2(q - 1)$.  
 If $q_0$ satisfies the following inequality: 
 {
 \begin{align*}
 q_0> \left((q-3)\begin{pmatrix}
   q-1\\\frac{q-1}{2}
 \end{pmatrix}+3\right)^2
\end{align*}
}then $P_k$ is an induced subgraph of $P(q_0)${, where $q\geq 5$}. 
\end{cor}

\begin{proof}
  $P_k$ is an induced subgraph of $C_{2(q-1)}$, and from Lemma \ref{lem:cycle1}, $C_{2(q-1)}$ is an induced subgraph of $P(q_0)$ under this condition.
\end{proof}


\begin{lem}\label{lem:oddCycle}
Let $k$ be an integer such that $k\leq q-1$.  
  If $q_0$ satisfies the following inequality: 
{  \[
    q_0>\left(2^q(q-2)\begin{pmatrix}
      q-1\\\frac{q-1}{2}
    \end{pmatrix}+3\right)^2
  \]
}then $C_{2k+1}$ is an induced subgraph of $P(q_0)$. 
\end{lem}

\begin{proof}
  From the following condition
  \begin{align*}
    {  
      q_0>\left(2^q(q-2)\begin{pmatrix}
        q-1\\\frac{q-1}{2}
      \end{pmatrix}+3\right)^2
  }    &>\left((q-3)\begin{pmatrix}
      q-1\\\frac{q-1}{2}
    \end{pmatrix}+3\right)^2,  
  \end{align*}
  $q_0$ satisfies the condition of Lemma \ref{lem:cycle1}.
  Therefore, there exist vertex subsets $A$ and $B$ of $P(q_0)$ as 
considered in Lemma \ref{lem:cycle1}, that is,
  \[
  A=\{ax\mid a\in \FF_q\},  B=\{ay\mid a\in \FF_q\},  
  \]
  and the induced subgraph of $A\cup B\setminus\{0\}$ is $C_{2(q-1)}$.

  We assume the existence of a vertex $z$ which is adjacent to $0$ and $x$, and not adjacent to other vertices in $A\cup B$.
  Let $\alpha$ be a primitive root of $\FF_q$.
  Then, $\alpha^{k-1}z$ is also adjacent to only $0$ and $\alpha^{k-1} x$ in $A\cup B$, and because $z$ is a quadratic residue, $z$ is adjacent to $\alpha^{k-1} z$.
  Therefore, the induced subgraph constructed by the vertex set 
  \[
    \{x,y,\alpha x,\alpha y,\alpha^2 x,\alpha^2 y,\ldots,\alpha^{k-2} x,\alpha^{k-2} y,\alpha^{k-1} x,\alpha^{k-1} z,z\}  
  \]
  is $C_{2k+1}$.
  Hence, we consider the condition of $q_0$ under which such a vertex $z$ exists.

  In the same manner as other lemmas, let $C=\{0,x\}$ and $D=(A\cup B)\setminus C$, and define $\tau(v)$ as the same in other Lemmas.
  Since all vertices in $A\cup B$ are not adjacent to $0$, the desired vertex exists when 
  \[
    \sum_{v\in \FF_{q_0}}\tau(v)=q_0+\sum_{\emptyset\neq S\subset A\cup B}(-1)^{|S\cap D|}\sum_{v\in \FF_{q_0}}\sigma(f_S(v))>0.  
  \]
  We evaluate the second term on the left-hand side.

  Let $\mathcal{S}$ be the power set of $A\cup B$. 
  Additionally, we consider a bijection $\varphi_a$ using an element $a \in \mathbb{F}_q$, as described in Lemma \ref{lem:bipartite}.
Following the same discussion as in Lemma \ref{lem:bipartite}, we obtain
\begin{align*}
  \sum_{\emptyset\neq S\subset A\cup B}(-1)^{|S\cap D|}\sum_{v\in \FF_{q_0}}\sigma(f_S(v))&=\sum_{\emptyset\neq S\subset A\cup B}(-1)^{|S\cap aD|}\sum_{v\in \FF_{q_0}}\sigma(f_S(v)).
\end{align*}
From the above,
\[
  \sum_{\emptyset\neq S\subset A\cup B}(-1)^{|S\cap D|}\sum_{v\in \FF_{q_0}}\sigma(f_S(v))=\sum_{\emptyset\neq S\subset A\cup B}\sum_{a\in \FF_q^*}\frac{(-1)^{|S\cap aD|}}{q-1}\sum_{v\in \FF_{q_0}}\sigma(f_S(v)).  
\]
Furthermore, 
\begin{align*}
  \sum_{{ \emptyset\neq S\subset  (A\cup B)\setminus\{0\}}}
  \sum_{a\in \FF_q^*}(-1)^{|S\cap aD|}&=\sum_{\emptyset\neq S\subset (A\cup B)\setminus\{0\}}\left(s_a(-1)^{s-1}+(q-1-s_a)(-1)^{s}\right)\\
  &=\sum_{\emptyset\neq S\subset (A\cup B)\setminus\{0\}}(q-1-2s_a)(-1)^{s},
\end{align*}  
  
where $|S|=s$ and $|S\cap A|=s_a$.
{ From Theorem \ref{thm:weil} and Theorem \ref{thm:Ananchuen},
\begin{align*}
  (-1)^{s}\sum_{v\in \FF_{q_0}}\sigma(f_S(v)+\sigma(f_{S\cup \{0\}}(v)))&\geq
  -\left|\sum_{v\in \FF_{q_0}}(\sigma(f_S(v))+\sigma(f_{S\cup \{0\}}(v)))\right|\\
  &\geq -(1+2(|S|-1)\sqrt{q_0}).
\end{align*}
}
Therefore,
{\begin{align*}
  \sum_{\emptyset\neq S\subset A\cup B}(-1)^{|S\cap D|}\sum_{v\in \FF_{q_0}}\sigma(f_S(v))\geq-\sum_{\emptyset\neq S\subset (A\cup B)\setminus\{0\}}\frac{(q-1-2s_a)(1+2(|S|-1)\sqrt{q_0})}{q-1}
\end{align*}
}and we obtain 
{\begin{align*}
  \sum_{v\in\FF_{q_0}}\tau(v)\geq q_0-\sum_{\emptyset\neq S\subset (A\cup B)\setminus\{0\}}\frac{|q-1-2s_a|(1+2(|S|-1)\sqrt{q_0})}{q-1}.
\end{align*}
}We consider the value of {$V=\sum_{\emptyset\neq S\subset (A\cup B)\setminus\{0\}}\frac{|q-1-2s_a|(1+2(|S|-1)\sqrt{q_0})}{q-1}$. }
Removing the absolute value of $V$, we obtain 
\begin{align*}
  V={-(-2\sqrt{q_0}+1)}+
  \left(\sum_{i=0}^{\frac{q-1}{2}}-\sum_{i=\frac{q+1}{2}}^{q-1}\right)\sum_{j=0}^{q-1}{\frac{q-1-2i}{q-1}(1+2(i+j-1)\sqrt{q_0})}\begin{pmatrix}
    q-1\\
    i
  \end{pmatrix}\begin{pmatrix}
    q-1\\
    j
  \end{pmatrix}.
\end{align*}
From Lemma \ref{lem:bipartite}, we obtain 
{\begin{align*}
  \left(\sum_{i=0}^{\frac{q-1}{2}}-\sum_{i=\frac{q+1}{2}}^{q}\right)\frac{q-1-2i}{q-1}(1+2(i-1)\sqrt{q_0})\begin{pmatrix}
    q-1\\i
  \end{pmatrix}
  &=\sqrt{q_0}(q-3)\begin{pmatrix}
    q-1\\\frac{q-1}{2}
  \end{pmatrix}+\begin{pmatrix}
    q-1\\\frac{q-1}{2}
  \end{pmatrix}
\end{align*}
}  and 
{  \begin{align*}
  \left(\sum_{i=0}^{\frac{q-1}{2}}-\sum_{i=\frac{q+1}{2}}^{q}\right)\frac{q-1-2i}{q-1}\begin{pmatrix}
    q-1\\i
  \end{pmatrix}&=\begin{pmatrix}
    q-1\\\frac{q-1}{2}
  \end{pmatrix}.
\end{align*}
  Note that this summation starts with $i=0$.
}
Therefore,
{\begin{align*}
  V&=2\sqrt{q_0}-1+\sum_{j=0}^{q-1}\left(\sqrt{q_0}(q-3)\begin{pmatrix}
    q-1\\\frac{q-1}{2}
  \end{pmatrix}-\begin{pmatrix}
    q-1\\\frac{q-1}{2}
  \end{pmatrix}\right)\begin{pmatrix}
    q-1\\j
  \end{pmatrix}+2j\sqrt{q_0}\begin{pmatrix}
    q-1\\\frac{q-1}{2}
  \end{pmatrix}\begin{pmatrix}
    q-1\\j
  \end{pmatrix}\\
  &=2\sqrt{q_0}\left(2^{q-1}(q-2)\begin{pmatrix}
    q-1\\\frac{q-1}{2}
  \end{pmatrix}+1\right)+2^{q-1}\begin{pmatrix}
    q-1\\\frac{q-1}{2}
  \end{pmatrix}-1.
\end{align*}
}
Based on the above, 
{
  \begin{align*}
    \sum_{v\in \FF_{q_0}}\tau(v)&\geq q_0-V\\
    &=q_0-2\sqrt{q_0}\left(2^{q-1}(q-2)\begin{pmatrix}
      q-1\\\frac{q-1}{2}
    \end{pmatrix}+1\right)-2^{q-1}\begin{pmatrix}
      q-1\\\frac{q-1}{2}
    \end{pmatrix}+1.
  \end{align*}}
Henceforth, {when $q\geq 3$,} if $q_0$ satisfies the following inequality: 
{\begin{align*}
  q_0&>\left(2^q(q-2)\begin{pmatrix}
    q-1\\\frac{q-1}{2}
  \end{pmatrix}+3\right)^2,
\end{align*}}
then $\sum_{v\in \FF_{q_0}}\tau(v)>0$ and the desired vertex exists. 
\end{proof}
\subsection{Non-trivial upper bound of the Paley-induced index for some graphs}

In this subsection, we provide a proof of statements \eqref{thm:indexBipartite} and \eqref{thm:indexOddCycle} in Theorem \ref{thm:main3} using the lemmas in subsection \ref{subsec:lem}.


\begin{proof}[Proof of Theorem \ref{thm:main3} \eqref{thm:indexBipartite}]
  From Corollary \ref{cor:bipartite}, Corollary \ref{cor:cycle1}, and Corollary \ref{cor:path1}, if
{  \[
    (q^{3^n})^2>\left((q^{3^m}-3)\begin{pmatrix}
      q^{3^m}-1\\\frac{q^{3^m}-1}{2}
    \end{pmatrix}+3\right)^2,
  \]
}  that is, 
{  \[
    n>\log_3\log_q\left((q^{3^m}-3)\begin{pmatrix}
      q^{3^m}-1\\\frac{q^{3^m}-1}{2}
    \end{pmatrix}+3\right),  
  \]
}  then $K_{k_1,k_2}, C_{2k_1}$ and $P_{k_1+k_2-1}$ are induced subgraphs of $P_n(q^2,3)$.
  Since the right-hand side is not an integer,
{  \[
    i_{P}(G)\leq \left\lceil \log_3\log_q\left((q^{3^m}-3)\begin{pmatrix}
      q^{3^m}-1\\\frac{q^{3^m}-1}{2}
    \end{pmatrix}+3\right)  \right\rceil. 
  \]
}\end{proof}




\begin{rem}
  As the order of the graph increases, 
  the evaluation of \eqref{thm:indexBipartite} in Theorem \ref{thm:main3} becomes stronger compared to the upper bound provided in Theorem \ref{thm:upperB}. 
  Specifically, when the order is $2(q^{3^m}-1)$, the evaluation becomes approximately half of the upper bound obtained from Theorem \ref{thm:upperB}. 
  The upper bounds obtained from these theorems and corollary are stronger than that of Theorem \ref{thm:upperB} when the order $k$ satisfies the following inequality:
  {\[(k-1)2^{k-2}\geq (q^{3^m}-3)\begin{pmatrix}
    q^{3^m}-1\\\frac{q^{3^m}-1}{2}
  \end{pmatrix}+3.\]}
\end{rem}

\begin{proof}[Proof of Theorem \ref{thm:main3} \eqref{thm:indexOddCycle}]
  From Lemma \ref{lem:oddCycle}, if
{  \[
    n>  \log_{3}\log_{q}\left(
      2^{q^{3^m}}(q^{3^m}-2)\begin{pmatrix}
        q^{3^m}-1\\\frac{q^{3^m}-1}{2}
      \end{pmatrix}+3
    \right),
  \]
}  then {$C_{2k_1+1}$} is an induced subgraph of $P_n(q^2,3)$.
  Since the right-hand side is not an integer, we obtain the desired inequality.
\end{proof}

\begin{rem}
  When 
{  \[
    k_12^{2k_1}\geq 
      2^{q^{3^m}}(q^{3^m}-2)\begin{pmatrix}
        q^{3^m}-1\\\frac{q^{3^m}-1}{2}
      \end{pmatrix}+3,
  \]
}  the bound obtained from \eqref{thm:indexOddCycle} in Theorem \ref{thm:main3} is stronger than that from Theorem \ref{thm:upperB}.
  Because
  \[
    \begin{pmatrix}
      q^{3^m}-1\\\frac{q^{3^m}-1}{2}
    \end{pmatrix}\leq \frac{2^{q^{3^m}-1}}{\sqrt{q^{3^m}}},
  \]
  when $q^{3^m}$ is sufficiently large, there exist many {$k_1$} that satisfy such a condition.
\end{rem}

\section{The value of the Paley index for cycles and paths}\label{sec:cycles and paths}
In this section, we provide trivial upper bounds for the Paley index. 
Also, we give a proof of Theorem \ref{thm:main3} \eqref{thm:subInd}. 

Similarly to the Paley-induced index, we can define the Paley index $\widetilde{i}_{P}$ as follows:
\[
\widetilde{i}_{P}(G)=\min\{n\in \NN\mid \mbox{$G$ is a subgraph of $P_n(q^2,3)$}\}. 
\]
{
Note that the Kneser index of any graph is $1$.  
This is because $K_{\NN,1}$ is a complete graph with infinitely many vertices.  
Therefore, we consider only the Paley index.
}

\begin{thm}\label{thm:upperBS}
  Let $G$ be a graph with order $k$. Then,
  \begin{align*}
    \widetilde{i}_{P}(G)&\leq \lceil \log_3\log_{q}k\rceil.
  \end{align*}
\end{thm}

\begin{proof}
  From Proposition 6.13 of \cite{MR3821579} (see also \cite{BT}), because the independence number of $P_n(q^2,3)$ is $q^{3^{n}}$ and the Paley graph is self-complementary, 
  $P_n(q^2,3)$ contains a complete graph with order $q^{3^{n}}$.
  Obviously, a complete graph with such order contains any graph with order $k$, where $k$ is any integer smaller than $q^{3^{n}}$.
  In other words, any graph with order $k$ is a subgraph of $P_n(q^2,3)$ when $n$ satisfies the following inequality:
  \[
    n\geq \log_3\log_qk.
  \]
  From the definition of $\widetilde{i}_{P}$, we obtain the desired upper bound.
\end{proof}

Next, we calculate the Paley index for cycles by utilizing the property of pancyclicity in the Paley graph.

\begin{df}[\cite{BONDY197180}]
  An undirected graph $G$ with order $n\geq 3$ is pancyclic if it contains a $k$-cycle as a subgraph for every $k\in \{3,4,\ldots,n\}$.
\end{df}

In \cite{TYamashita}, 
they proved a sufficient condition for pancyclicity. 
For any vertex $x$ in a graph, $N(x)$ represents the neighborhood set of $x$ and $d(x)$ represents the degree of $x$. 
\begin{thm}[\cite{TYamashita}]\label{thm:tpanc}
  Let $G$ be a 2-connected graph of order $n\geq 6$. 
  Suppose that $|N(x)\cup N(y)|+d(z)\geq n$ for every triple independent vertices $x,y,z$ of $G$.
  Then $G$ is pancyclic or isomorphic to the complete bipartite graph $K_{\frac{n}{2},\frac{n}{2}}$.
\end{thm}

From the properties of the Paley graph, we can easily prove that 
 for any triple independent vertices $x,y,z$ in the Paley graph with order $q\geq 6$, 
 \[|N_G(x)\cup N_G(y)|+d_G(z)=\frac{5q-5}{4}\geq q.\]
Therefore, from Theorem \ref{thm:tpanc}, we easily demonstrate the pancyclicity of the Paley graph.
From this, we can determine the Paley index of cycle graphs and paths.


\begin{proof}[Proof of Theorem \ref{thm:main3} \eqref{thm:subInd}]
  Because $P_n(q^2,3)$ is pancyclic, if $(q^2)^{3^n}\geq k$, in other words 
  \[
    n\geq \left\lceil\log_3\frac{1}{2}(\log_qk)\right\rceil,  
  \]
  then $P_n(q^2,3)$ contains $C_k$ as a subgraph.
  Hence, we obtain $\widetilde{i}_{P}(C_k)\leq \lceil\log_3\frac{1}{2}(\log_qk)\rceil$.
  Additionally, if $(q^2)^{3^n}<k$ then obviously $P_n(q,3)$ cannot contain $C_k$.
  This implies $\widetilde{i}_{P}(C_k)\geq \lceil\log_3\frac{1}{2}(\log_qk)\rceil$, and we have $\widetilde{i}_{P}(C_k)=\log_3\frac{1}{2}(\log_qk)$.
  Because $P_k$ is a subgraph of $C_k$, we have also derived $\widetilde{i}_{P}(P_k)\leq \lceil\log_3\frac{1}{2}(\log_qk)\rceil$ and completing the proof.
\end{proof}


\section{Remarks and questions}\label{sec:rem}

\begin{rem}
$X_{K_{\mathbb{N},2}}(\bullet)$ distinguish $G_{1}$ from $G_{2}$ in Figure \ref{Fig: G1 and G2}. 
Indeed there is a homomorphism from $G_{2}$ to $K_{\NN,2}$ given 
as follows, 
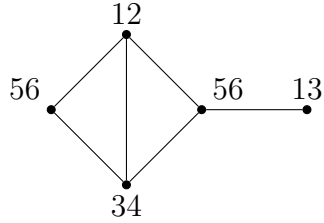
\begin{center}
\begin{tikzpicture}[scale=2]
\draw (-.5, .5) node[v](1){} node[above]{12};
\draw (-1, 0) node[v](2){} node[above left]{56};
\draw (-.5,-.5) node[v](3){} node[below]{34};
\draw (0,0) node[v](4){} node[above right]{56};
\draw (.7,0) node[v](5){} node[above]{13};
\draw (1)--(2)--(3)--(4)--(1)--(3);
\draw (4)--(5);
\end{tikzpicture}
\end{center}
while there is no homomorphism from $G_1$ to $K_{\NN,2}$ 
whose image consists of $\{12,13,34,56,56\}$. 
\end{rem}

\begin{rem}
The integer sequence of the number of connected graphs in $\mathcal{P}_{n}^{(2)}$ begins from 
\begin{align*}
1, 2, 5, 12, 33, 103, 333, 1183, 4442, 17576, \dots
\end{align*}
It is recorded as A076864 in OEIS \cite{OEIS}. 

The integer sequence consisting of cardinalities of $\mathcal{P}_{n}^{(2)}$ (or equivalently, the dimensions homogeneous parts of $\Sym^{(2)}$) starts from 
\begin{align*}
1, 3, 8, 23, 66, 212, 686, 2389, 8682, 33160, \dots
\end{align*}
and it is A050535 in OEIS \cite{OEIS}. 
\end{rem}

{Question:} 
Does $\mathcal{A}_{\bullet}^{(2)}$ distinguish trees? 
If so, we can conclude that $X_{K_{\mathbb{N},2}}(\bullet)$ distinguish trees by Theorem \ref{power sum expansion}. 

{Observation:}
Let $T$ be a tree. 
Suppose that $\lambda \in \mathcal{A}_{T}^{(2)}$ is also a tree. 
Then every leaf of $T$ corresponds with a leaf edge of $\lambda$. 
In particular, assume that the number of leaves of $T$ and $\lambda$ are the same. 
Then $T$ may be reconstructed by removing a leaf from $\lambda$. 

\medskip
{Question:} 
Is there a pair $G_{1}, G_{2}$ of non-isomorphic graphs 
such that $X_{K_{\mathbb{N},2}}(G_{1}) = X_{K_{\mathbb{N},2}}(G_{2})$? 

\begin{rem}
Let $\mathcal{T}$ be the set of all trees and 
$i$ be the Kneser-functional index of $\mathcal{T}$. 
Conjecture \ref{conj:tree} means that $i=1$. 

\medskip
{Question:} 
Can we give an upper bound for $i$?
\end{rem}

\section*{Acknowledgments}

The authors would like to thank the anonymous reviewers for their 
beneficial comments on an earlier version of the manuscript. 
The authors were supported by JSPS KAKENHI (22K03277, 22K03398, 22K13885). 


\end{document}